\documentclass[reqno,A4paper]{amsart}
\usepackage{amsmath}
\usepackage{amssymb}
\usepackage{amsthm}
\usepackage{enumerate, verbatim}
\usepackage{mathrsfs}
\usepackage{eqlist}
\usepackage{array}
\usepackage{float}
\usepackage{hyperref}
\usepackage{tikz}
\usepackage{tikz-qtree}
\usepackage{diagbox}
\usepackage[title]{appendix}
\theoremstyle{plain}
\newtheorem{theorem}{Theorem}[section]
\newtheorem{proposition}[theorem]{Proposition}
\newtheorem{lemma}{Lemma}[section]

\theoremstyle{definition}

\newtheorem{remark}{\textup{Remark}}

\numberwithin{equation}{section}

\begin{document}

\title[A classification of Markoff-Fibonacci $m$-triples]{A classification of Markoff-Fibonacci $m$-triples}

\author[D. Alfaya, L. A. Calvo, A. Mart\'inez de Guinea, J. Rodrigo \and A. Srinivasan]{D. Alfaya*  ***, L. A. Calvo**, A. Mart\'inez de Guinea***, J. Rodrigo* \and A. Srinivasan**}

\newcommand{\acr}{\newline\indent}

\address{\llap{*\,}Department of Applied Mathematics,\acr
ICAI School of Engineering, Comillas Pontifical University,\acr
C/Alberto Aguilera 25, 28015 Madrid,\acr Spain.}
\email{dalfaya@comillas.edu, jrodrigo@comillas.edu}

\address{\llap{**\,}Department of Quantitative Methods,\acr
ICADE, Comillas Pontifical University,\acr
C/Alberto Aguilera 23, 28015 Madrid,\acr Spain.}
\email{lacalvo@comillas.edu, asrinivasan@icade.comillas.edu}

\address{\llap{***\,}Institute for Research in Technology,\acr
ICAI School of Engineering,\acr
Comillas Pontifical University,\acr
C/Santa Cruz de Marcenado 26, 28015 Madrid,\acr Spain.}
\email{dalfaya@comillas.edu, alex.m.guinea@alu.comillas.edu}

\thanks{\textit{Acknowledgements}. This research was supported by project CIAMOD (Applications of computational methods and artificial intelligence to the study of moduli spaces, project PP2023\_9) funded by Convocatoria de Financiaci\'on de Proyectos de Investigaci\'on Propios 2023, Universidad Pontificia Comillas, and by grants PID2022-142024NB-I00 and PID2019-104735RB-C42 funded by MCIN/AEI/10.13039/501100011033. We would like to thank Sergio Herreros, Daniel S\'anchez and Jaime Pizarroso for useful discussions.}


\begin{abstract}
We classify all solution triples with Fibonacci components to the equation $a^2+b^2+c^2=3abc+m,$ for positive $m$. We show that for $m=2$ they are precisely $(1,F(b),F(b+2))$, with even $b$; for $m=21$, there exist exactly two Fibonacci solutions $(1,2,8)$ and $(2,2,13)$ and for any other $m$ there exists at most one Fibonacci solution, which, in case it exists, is always minimal (i.e. it is a root of a Markoff tree).
Moreover, we show that there is an infinite number of values of $m$ admitting exactly one such solution.
\end{abstract}

\maketitle

\setcounter{section}{0}
\section{Introduction}

Markoff triples constitute a fascinating area of research in number theory. They are the solution triples in positive integers of the Markoff equation
$$
x^2 + y^2 + z^2 = 3xyz.
$$  Originally introduced by A. A. Markoff in his work \cite{M1,M2}, in the context of minima of quadratic forms, these triples exhibit interesting properties; for instance, they are all distributed through a tree structure, called the Markoff tree. 

In recent years, many authors have studied generalizations of this equation (\cite{Mor}, \cite{GS}). In \cite{SC}, Markoff $m$-triples are introduced as positive integer solutions of the $m$-Markoff equation
\begin{equation}\label{Markoffeq}
x^2+y^2+z^2=3xyz+m,
\end{equation}
where $m$ is a positive integer. By the symmetry of the equation, any permutation of a solution is also a solution, and hence the Markoff $m$-triples $(x,y,z)$ are assumed to be ordered with $0<x\le y \le z$. These triples are also organized into tree structures. However, in this case, multiple trees or none could exist. The authors in \cite{SC} showed that the number of trees, for every $m,$ is equal to the number of Markoff $m$-triples $(x,y,z)$ that are minimal, that is to say, those that satisfy the inequality
\begin{equation}\label{Minimals}
z \geq 3xy.
\end{equation}
Moreover, these minimal triples are the roots of the trees generated by Vieta transformations.

Several authors have studied various generalized Markoff equations, looking for solutions of these equations from sequences such as Fibonacci, Pell, Lucas, and so on (see \cite{AL},\cite{B},\cite{GGL},\cite{H},\cite{HT},\cite{HST}, \cite{RSP}, \cite{R} and \cite{T}). In this paper, we study Markoff-Fibonacci $m$-triples for $m>0$, which are Markoff $m$-triples such that all the components are Fibonacci numbers. The case of Markoff-Fibonacci triples $(m=0),$ was previously studied in \cite{LS}, where it was found that the only Fibonacci solutions to the Markoff equation were triples of the form $(1,F(b),F(b+2)),$ where $b$ is odd.

We found that an analogous family of Fibonacci solutions exists for the generalized $m$-Markoff equation with $m=2$, containing all possible non-minimal solutions. Concretely, if $b$ is an even number, then $(1,F(b),F(b+2))$ is always a solution of the $m$-Markoff equation for $m=2$. This family forms the superior branch of the  Markoff $2$-tree, as shown in Figure \ref{2-markovtree} below.

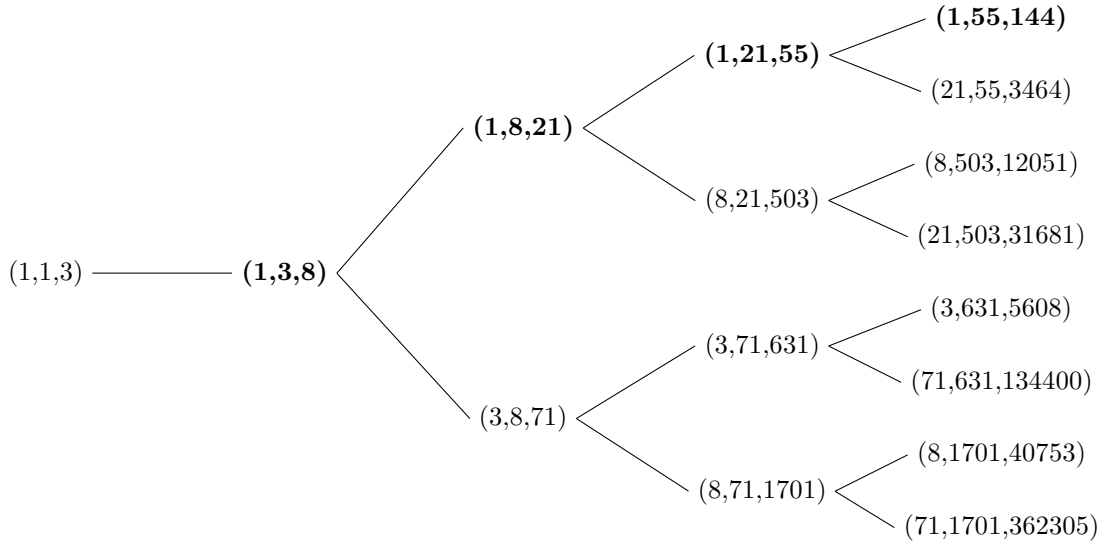
\begin{figure}[H]
\centering
\begin{tikzpicture}[grow'=right,level distance=1.25in, sibling distance=0.15in,
    calculate width/.code={
        \pgfmathsetmacro{\tempwidth}{width("#1")*3}
        \tikzset{every tree node/.append style={text width=\tempwidth pt}}}]
\Tree 
[. (1,1,3)
    [. \textbf{(1,3,8)}
        [.\textbf{(1,8,21)}
            [.\textbf{(1,21,55) }
                    [.\textbf{(1,55,144)} ]
                    [.(21,55,3464) ]
            ]
            [.(8,21,503)
                    [.(8,503,12051) ]
                    [. (21,503,31681) ]
            ]
        ]
        [.(3,8,71)
            [.(3,71,631) 
                    [.(3,631,5608)  ]
                    [.(71,631,134400) ]
            ]
            [.(8,71,1701) 
                    [.(8,1701,40753)  ]
                    [.(71,1701,362305) ]
            ]
        ] 
    ]
]
\end{tikzpicture}
\caption{Beginning of the Markoff $2$-tree with minimal $2$-triple $(1,1,3)$. The sequence of non-minimal Markoff-Fibonacci $2$-triples is represented in bold.}
\label{2-markovtree}
\end{figure}

 Every other $m\ne 2$, except for $m=21$, has at most one possible Fibonacci solution to the $m$-Markoff equation, which is always a minimal triple (i.e., a root of a Markoff $m$-tree). This is summarised in the following theorem yielding a classification of all Markoff-Fibonacci $m$-triples (see Proposition \ref{prop:non-minimal} and Proposition \ref{prop:minimal}).

\begin{theorem}
\label{thm:intro}
For each $m>0,$ there exists at most one ordered solution to the equation
$$x^2+y^2+z^2=3xyz+m$$
composed of Fibonacci numbers, except in the following cases.
\begin{itemize}
    \item If $m=2$, the Fibonacci solutions are $(1,F(b),F(b+2))$ for each even $b\ge 2$.
    \item If $m=21$, the Fibonacci solutions are the minimal triples $(1,2,8)$ and $(2,2,13)$.
\end{itemize}
Moreover, there exists an infinite number of $m>0$ admitting exactly one Markoff-Fibonacci $m$-triple and such triple is always minimal.
\end{theorem}

The paper is structured as follows. Section \ref{section:preliminary} provides certain inequalities satisfied by Fibonacci numbers which will be useful for some of the proofs. In Section \ref{section:non-minimal} a characterisation of Markoff $m$-triples $(F(a),F(b),F(c))$  for a positive $m$ is obtained. Section \ref{section:minimal} is devoted to the full classification of all Markoff-Fibonacci triples. We first characterise non-minimal $m$-triples and then focus on the proof of the uniqueness of minimal solutions of the $m$-Markoff equation. This proof is split into three parts. First, using analytic techniques, we establish that if $(F(a),F(b),F(c))$ and $(F(a'), F(b'), F(c'))$ are two ordered minimal Markoff-Fibonacci $m$-triples with $c\ge c'$, then $c'=c-1$ or $c=c'$. Then, we show that for a given $m$, in either of the cases above, if $c$ is large enough, then two such different triples cannot exist. Finally, we verify computationally that the result holds for the remaining finite values of $c$ (see Appendix \ref{section:appendix1}).

%
%

\section{Some inequalities of Fibonacci numbers}
\label{section:preliminary}

Let us start by recalling some facts about Fibonacci numbers which will be useful in the study of Fibonacci solutions of the $m$-Markoff equation.

The well-known recursion formula for the Fibonacci sequence is
given by $F(0)=0$, $F(1)=1$ and $F(n+1)=F(n)+F(n-1)$, for all 
$n\ge 1$. Throughout this paper, we denote
$$\varphi=\frac{1+\sqrt{5}}{2}, \quad \quad \bar{\varphi}=\frac{1-\sqrt{5}}{2}$$
so that for each $n\ge 0$, the $n$-th Fibonacci number can be written through Binet's formula as
\begin{equation}\label{eq:Binet}
F(n)=\frac{\varphi^n-\bar{\varphi}^n}{\sqrt{5}}\, .\end{equation}

The following bounds on quotients of Fibonacci numbers will be used in later sections to compute bounds on Markoff-Fibonacci $m$-triples.

\begin{lemma}
\label{lemma:boundFibonacciQuotient}
Let $a$ be an integer and let $N>0$ be an integer. Let
$$k_{N,a}= \min \left(\frac{F(N)}{F(N+a)},\frac{F(N+1)}{F(N+1+a)}\right),\hspace{0.5cm} K_{N,a}=\max \left(\frac{F(N)}{F(N+a)},\frac{F(N+1)}{F(N+1+a)}\right).$$
Then, for each $n\ge N$, the following inequalities hold.
$$k_{N,a}\le \frac{F(n)}{F(n+a)}\le K_{N,a}.$$
\end{lemma}

\begin{proof}
Let $k,K\in \mathbb{R}$ be any pair of numbers such that for some $N$
$$k\le \frac{F(N)}{F(N+a)}, \frac{F(N+1)}{F(N+a+1)} \le K.$$
We will show that
$$kF(n+a)\le F(n) \le KF(n+a)$$
for each $n\ge N$. We achieve this by induction on $n$. The result clearly holds for $n=N$ and $n=N+1$ by hypothesis. Let $n\ge N+2$ and assume that the statement is true for all $n'$ with $N\le n'< n$. In particular, it follows that
$$kF(n+a-2)\le F(n-2) \le KF(n+a-2)\,,$$
$$kF(n+a-1)\le F(n-1) \le KF(n+a-1)\,.$$
Adding both expressions yields
\begin{multline*}kF(n+a)=kF(n+a-1)+kF(n+a-2)\le F(n)=F(n-1)+F(n-2)\\ \le KF(n+a-1)+KF(n+a-2)=KF(n+a)\, .\end{multline*}
The lemma now follows on taking $k=k_{N,a}$ and $K=K_{N,a}$.
\end{proof}

\begin{remark}
\label{rmk:FibonacciQuotientBounds}
In Appendix \ref{section:appendix2}, Tables \ref{table:1} and \ref{table:2} have been computed by rounding down and up respectively the values of $k_{N,a}$ and $K_{N,a}$ provided by Lemma \ref{lemma:boundFibonacciQuotient} for small values of $N$ and $a$. We will use these concrete lower and upper bounds for the quotients $F(n)/F(n+a)$ in the proofs of some inequalities in the following sections.
\end{remark}

\begin{remark}
\label{rmk:Vadja}
For the sake of  exposition, we recall the following identity on Fibonacci numbers, usually known as Vajda's identity \cite{V} which will be used in several proofs throughout the paper:
$$F(n+i)F(n+j)-F(n)F(n+i+j)=(-1)^{n}F(i)F(j),$$
for any positive integers $i,j$ and $n.$ In particular, taking $n=1,\,\, i=a$ and $j=b-1$ we obtain the famous equality $F(a+b)=F(a+1)F(b)+F(a)F(b-1)$.

\end{remark}

\begin{lemma}
\label{lemma:F(n+m)<=3F(n)F(m)}
Let $a,b,c\ge 2$. Then
\begin{align}
F(c)\le  3F(a)F(b)& \quad \text{if and only if} \quad c\le a+b, \label{eq:F(n+m)<=3F(n)F(m)1}\\
F(c)> 3F(a)F(b)& \quad \text{if and only if} \quad c\ge a+b+1, \text{and} \label{eq:F(n+m)<=3F(n)F(m)2}\\
F(c)= 3F(a)F(b)& \quad \text{if and only if} \quad a=b=2,\, c=4\label{eq:F(n+m)<=3F(n)F(m)3}\, .
\end{align}
\end{lemma}
\begin{proof}
 By Remark \ref{rmk:Vadja}, substituting $F(a+1)=F(a)+F(a-1)$  we obtain
\begin{equation}\label{ine_lema3}
F(a+b)=F(a)F(b)+F(a-1)F(b)+F(a)F(b-1)\le 3F(a)F(b)\, .
\end{equation}
As $F(c)$ is increasing in $c$, this gives \eqref{eq:F(n+m)<=3F(n)F(m)1}.

Suppose that $c\ge a+b+1$. Since $F(c)\ge F(a+b+1)$, to prove \eqref{eq:F(n+m)<=3F(n)F(m)2}  it is enough to show that 
\begin{equation}\label{condlema3.1}
F(a+b+1)> 3F(a)F(b).
\end{equation}
As before, developing the left-hand side of \eqref{condlema3.1} using Vajda's identity yields
\begin{multline*}F(a+b+1)=F(a+1)F(b+1)+F(a)F(b)=2F(a)F(b)+F(a)F(b-1)+F(a-1)F(b)+F(a-1)F(b-1)\\=2F(a)F(b)+F(a)F(b-1)+2F(a-1)F(b-1)+F(a-1)F(b-2)\,.
\end{multline*}
On the other hand, working on the right-hand side of \eqref{condlema3.1} we obtain
\begin{multline*}3F(a)F(b)=2F(a)F(b)+F(a)F(b)=2F(a)F(b)+F(a)F(b-1)+F(a)F(b-2)=\\2F(a)F(b)+F(a)F(b-1)+F(a-1)F(b-2)+F(a-2)F(b-2)\, .
\end{multline*}
Comparing both sides, we obtain that $2F(a-1)F(b-1)> F(a-2)F(b-2)$, which is clearly true for $a\geq 2$ and $b\geq 2$. Finally, we study the equality case. By \eqref{eq:F(n+m)<=3F(n)F(m)2}, we must have $c\le a+b$. As $a+b\ge 4$, then for each $c<a+b$ it holds that
$F(c)<F(a+b)\le 3F(a)F(b)$
consequently $c=a+b$. The equality in \eqref{ine_lema3} is only attained if $F(a-1)=F(a)$ and $F(b-1)=F(b)$. That only occurs in the case $a=b=2$, proving \eqref{eq:F(n+m)<=3F(n)F(m)3}.
\end{proof}

%
%

\section{Markoff-Fibonacci \texorpdfstring{$m$}{m}-triples}
\label{section:non-minimal}

Henceforth, we shall denote
$$m(a,b,c)=F(a)^2+F(b)^2+F(c)^2-3F(a)F(b)F(c)$$
so that $(F(a),F(b),F(c))$ is a Markoff-Fibonacci $m$-triple if and only if $m=m(a,b,c)>0$. In this section, we derive conditions on $(a,b,c)$ to ensure $m(a,b,c)>0$.

Since we want to study ordered triples $(F(a),F(b),F(c))$ composed of positive numbers, without loss of generality, we will assume from now on that $2\le a \le b \le c$.

\begin{lemma}
\label{lemma:minimalPositive}
If $2\le a\le b\le c$ with $c\ge a+b+1$, then $m(a,b,c)>0$.
\end{lemma}

\begin{proof}
Since
$$m(a,b,c)=F(a)^2+F(b)^2+F(c)\left(F(c)-3F(a)F(b)\right)\, ,$$
and by Lemma \ref{lemma:F(n+m)<=3F(n)F(m)}, $F(c)>3F(a)F(b)$, then all of the terms in the previous factorization are positive, and thus $m(a,b,c)>0$.
\end{proof}

\begin{lemma}\label{Javier_non-minimal} If $3\le a\le b\le c\le a+b,$ then $m(a,b,c)<0$.
\end{lemma}

\begin{proof}
As $F(x)$ is an increasing function, $1<F(a)\leq F(b)\leq F(c)\leq F(a+b).$
Consider the parabola $f(x)=x^2-3F(a)F(b)x+F(a)^2+F(b)^2.$  Since $f(x)$ is an upward-opening parabola, its absolute maximum in
$F(b)\leq x\leq F(a+b)$ is attained at one of the endpoints. Therefore, to prove that $m(a,b,c)=f(F(c))<0,$ it suffices to prove that $f(F(b))$ and $f(F(a+b))$ are both negative.

\noindent First, we prove that $f(F(b))<0.$ Indeed,
$$f(F(b))=F(a)^2+2F(b)^2-3F(a)F(b)^2\le F(b)^2(3-3F(a))<0.$$
Let us prove that $f(F(a+b))<0$. We have $F(a)=F(a-1)+F(a-2)\le 2F(a-1)$. Thus, $F(2a)=F(a)^2+2F(a)F(a-1)\ge 2F(a)^2$. Equation \eqref{ine_lema3} implies $F(2a)-3F(a)^2<-1$ for $a\ge 3$ because $F(a-1)<F(a)$. Therefore, in the case $a=b$,
$$f(F(2a))=m(a,a,2a)=2F(a)^2+F(2a)(F(2a)-3F(a)^2)<2F(a)^2-F(2a)\le 0.$$
Thus, we may assume that $a\ne b$. From Remark \ref{rmk:Vadja} we have $\frac{F(a+b)}{F(a)F(b)}=\frac{F(a+1)}{F(a)}+\frac{F(b-1)}{F(b)}\le \frac{1}{0.6}+0.6667<2.4$. Using Tables \ref{table:1} and \ref{table:2} it follows that
$\frac{1}{F(a)}+\frac{F(a+b)}{F(a)F(b)}<3$ and thus 
$F(b)+F(a+b)<3F(a)F(b)$. Hence
$$F(a)^2+F(b)^2<F(b)F(b+1)<(3F(a)F(b)-F(a+b))F(a+b)$$
giving $f(F(a+b))=m(a, b, a+b)<0$.
\end{proof}

\begin{lemma}\label{discard_luis} If $b\ge 2,$ then $m(2,b,b+1) \leq 0$, and the equality is only attained if $b=2$.
\end{lemma}
\begin{proof}
This is equivalent to proving
$$1+F(b)^2+F(b+1)^2\leq 3F(b)F(b+1).$$
Since $b\ge 2$, then
$$1+F(b)^2+F(b+1)^2=F(1)^2+F(b)^2+F(b+1)^2\le \sum_{k=1}^{b+1} F(k)^2 = F(b+1)F(b+2)$$
and, by Remark \ref{rmk:FibonacciQuotientBounds}, $F(b+2)\le 3F(b)$, with equality attained at $b=2$, so the lemma follows.
\end{proof}

As a consequence of the previous lemmata, we obtain a complete classification of triples\\ $(F(a),F(b),F(c))$ that are $m$-triples for a positive $m$.

\begin{proposition}
\label{prop:positivity}
Suppose that $2\le a\le b\le c$. Then $m(a,b,c)>0$ if and only if $a+b+1\le c$ or $(a,b,c)=(2,b,b+2),$ for some even $b\ge 2$.
\end{proposition}

\begin{proof}
If $c\ge a+b+1,$ then $m(a,b,c)>0$ by Lemma \ref{lemma:minimalPositive}. If $(a,b,c)=(2,b,b+2)$, Vajda's identity (see Remark \ref{rmk:Vadja}) for $(i=1, j=1, n=b)$ implies
\begin{multline}
\label{eq:2-non-minimal}
m(2,b,b+2)=1+F(b)^2+F(b+2)^2-3F(b)F(b+2)
=1+(F(b+2)-F(b))^2-F(b)F(b+2)\\=1+F(b+1)^2-F(b)F(b+2)=1+(-1)^b\,,
\end{multline}
which is positive if and only if $b$ is even. If $b$ is odd, we obtain a branch of the Markoff tree ($m=0$) described in the article \cite{LS}. For $b>2$ even, we obtain the superior branch of the $2$-Markoff tree with minimal triple $(1,1,3)$, represented in Figure \ref{2-markovtree}.

Suppose now that $c\le a+b$ and that $(a,b,c)\ne (2,b,b+2)$. If $a \ge 3$, then Lemma \ref{Javier_non-minimal} shows that $m(a,b,c)<0$. If $a=2$ and $c\ne b+2$ but $c\le a+b=b+2$, then either $(a,b,c)=(2,b,b)$ or $(a,b,c)=(2,b,b+1)$. If $(a,b,c)=(2,b,b),$ we obtain
$$m(a,b,c)=m(2,b,b)=1+2F(b)^2-3F(b)^2=1-F(b)^2\le 0.$$
If $(a,b,c)=(2,b,b+1)$ then Lemma \ref{discard_luis} shows that $m(a,b,c)=m(2,b,b+1)\le 0$.
\end{proof}

%
%

\section{Proofs of main results}
\label{section:minimal}

Using Proposition \ref{prop:positivity}, we obtain a specific characterisation of minimal and non-minimal Markoff-Fibonacci $m$-triples, when $m>0$. Recall that a Markoff $m$-triple $(x,y,z)$ with $m>0$ is called minimal (c.f. \cite[Definition 2.2]{SC}) if $z\ge 3xy\,.$

\begin{proposition}
\label{prop:minimalBound}
Let $a\ge 2$. Then for $a\le b\le c$, the $m$-triple $(F(a),F(b),F(c))$, with $m>0$, is minimal if and only if either $c\ge a+b+1$ or $a=b=2$  and $c= 4$.
\end{proposition}

\begin{proof}
By Proposition \ref{prop:positivity}, the values of $(a,b,c)$ for which $m(a,b,c)>0$ are those with $c\ge a+b+1$ or $(a,b,c)=(2,b,b+2),$ for some even $b\ge 2$. Amongst those triples, the minimal ones are those which also satisfy $F(c)\ge 3F(a)F(b)$. By Lemma \ref{lemma:F(n+m)<=3F(n)F(m)}, these are precisely the triples such that $c\ge a+b+1$ or $(a,b,c)=(2,2,4)$.
\end{proof}

Using this proposition, we will split the problem of classifying the Fibonacci solutions to the $m$-Markoff equation \eqref{Markoffeq} into classifying the non-minimal and minimal triples, starting by the classification of non-minimal $m$-triples.

\subsection{Non-minimal case}
\begin{proposition} \label{prop:non-minimal}
Every non-minimal Markoff-Fibonacci $m$-triple with $m>0$ is a $2$-triple of the form $(1,F(b),F(b+2)),$ where $b>2$ is an even number.
\end{proposition}

\begin{proof}
By Proposition \ref{prop:positivity}, any Markoff-Fibonacci $m$-triple $(F(a),F(b),F(c))$ with $m>0$ must either satisfy $c\ge a+b+1$, or be of the form $(1,F(b),F(b+2))$ for some even $b\ge 2$. By Proposition \ref{prop:minimalBound}, non-minimal $m$-triples satisfy $c\le a+b$ with $(a,b,c)\ne (2,2,4)$. Thus, all non-minimal Markoff-Fibonacci $m$-triples must be of the form $(1,F(b),F(b+2))$ with even $b\ge 3$ and, by \eqref{eq:2-non-minimal}, they are all $2$-triples. In particular, they form a branch of the $2$-Markoff tree spanned by $(1,3,8)$.
\end{proof}

\begin{remark}
From \cite{SC}, we know that for $m=2$ there exists exactly one minimal $2$-triple, which is the only remaining Fibonacci solution to the $m$-Markoff equation for $m=2$, namely,  $(1,1,3)=(F(2),F(2),F(4))$.
\end{remark}

\subsection{Minimal case}

We will now focus on the minimal Markoff-Fibonacci $m$-triples, with $m>0$. The following proposition shows that contrary to the non-minimal Fibonacci triples, which only exist for $m=2$, minimal Markoff-Fibonacci $m$-triples exist for an infinite number of values of $m$.

\begin{proposition}
\label{prop:infinite}
There exists an infinite number of values of $m>0$ for which the equation (\ref{Markoffeq})
admits a solution composed of Fibonacci numbers. Moreover, these solution triples are minimal.
\end{proposition}

\begin{proof}
By the characterisation of minimal $m$-triples from Proposition \ref{prop:minimalBound} and by Proposition \ref{prop:positivity}, we know that for each $a$ and $b$ at least 2 and each $c\ge a+b+1$, the triple $(F(a),F(b),F(c))$ is an $m$-triple for some $m>0$ and it is minimal. Observe that if we fix $a$ and $b$ then
$$m(a,b,c)=F(a)^2+F(b)^2+F(c)^2-3F(a)F(b)F(c)$$
is a quadratic equation in $F(c)$ which is positive and strictly increasing in $c$ for each $c\ge a+b+1$ because $F(a+b+1)>3F(a)F(b)>\frac{3}{2}F(a)F(b)$ due to Lemma \ref{lemma:F(n+m)<=3F(n)F(m)}. This shows that, for each $a$ and $b$, there is an infinite number of different values of $m>0$ and $c$ such that $(F(a),F(b),F(c))$ is a minimal $m$-triple.
\end{proof}

Next, we will show that, except for $m=21$, if $m>0$ admits a minimal Markoff-Fibonacci $m$-triple, it is unique up to order. We will devote the rest of the paper to proving such a claim in full generality.

Through the rest of the section, we will assume that $(F(a),F(b),F(c))$ and $(F(a'),F(b'),F(c'))$ are two different ordered Markoff-Fibonacci triples for the same $m>0$, i.e.,  $2\le a\le b\le c$, $2\le a'\le b'\le c'$ with $(a,b,c)\ne (a',b',c')$ and
$$m(a,b,c)=m(a',b',c')=m$$
and, without loss of generality, we may assume that $c\ge c'$.

If we fix an upper bound for $c$, there only exists a finite number of pairs of such triples. The following Lemma has been checked computationally (see Appendix \ref{section:appendix1} for the implementation details) and shows that the claim holds if $c$ is small.

\begin{lemma}
\label{lemma:computationalCheck}
Suppose that $2\le a\le b\le c<20$ and $2\le a'\le b'\le c'\le c<20$. If $m=m(a,b,c)=m(a',b',c')>0$, then either
\begin{itemize}
\item $a=a'$, $b=b'$, and $c=c'$ or,
\item $(a,b,c)=(3,3,7)$ and $(a',b',c')=(2,3,6)$, with $m(a,b,c)=m(a',b',c')=21$ or,
\item $m=2$ and $(a,b,c)$ and $(a',b',c')$ are of the form $(2,b,b+2)$, as described by Proposition \ref{prop:non-minimal}.
\end{itemize}
\end{lemma}

The rest of this work will focus on showing that it is impossible to find two triples $(F(a),F(b),F(c))$ and $(F(a'),F(b'),F(c'))$ with $m(a,b,c)=m(a',b',c')$ if $c\ge 20$. We will split the proof into two cases. First, we will study pairs of triples with $c>c'$ and we will prove that the only possible pair of triples is the known example from Lemma \ref{lemma:computationalCheck}, namely,  $m(3,3,7)=m(2,3,6)=21$. Then we will analyse the case $c=c'$ and prove that no pair of distinct Markoff-Fibonacci $m$-triples can exist with the same maximal element. Let us start with a useful monotonicity lemma.
\begin{lemma}
\label{lemma:increasing}
Suppose that $2\le a\le b\le c$ and $c\ge 5$. Suppose that $2\le a\le a'\le c$ and $b\le b'\le c$. Then
$$m(a,b,c)\ge m(a',b',c)$$
and the equality holds if and only if $a=a'$ and $b=b'$. In particular, if $(F(a),F(b),F(c))$ is an ordered minimal Markoff-Fibonacci $m$-triple with $m>0$, then
$$m(2,2,c)\ge m(a,b,c) \ge m(a,c-a-1,c).$$
\end{lemma}

\begin{proof}
The parabolas
$$p(x)=x^2-3F(b)F(c)x+F(b)^2 +F(c)^2 \quad \text{and} \quad q(x)=x^2 -3F(a)F(c)x+F(a)^2+F(c)^2$$
have their vertices at $x=\frac{3}{2}F(b)F(c)> F(c)$ and $x=\frac{3}{2}F(a)F(c)> F(c)$ respectively, so they are both strictly decreasing for $0\le x\le F(c)$ for all values of $a$ and $b$. As the function mapping $n$ to $F(n)$ is increasing, then for any given fixed $c$ the map
$$m_c : (a,b) \mapsto m(a,b,c)=F(a)^2+F(b)^2+F(c)^2-3F(a)F(b)F(c)$$
is strictly decreasing both in $a$ and $b$ in the region $2\le a,b\le c$. Given $a,b,c$ such that $(F(a),F(b),F(c))$ is a minimal triple, Proposition \ref{prop:minimalBound} implies that $a+b\le c-1$ so $2\le a\le b\le c-1-a<c$ and, therefore
$$m_c(a,b)\ge m_c(a,c-1-a) \quad \text{and} \quad m_c(a,b)\le m_c(2,b)\le m_c(2,2)\, .$$
\end{proof}

\begin{remark}
An analogue of Lemma \ref{lemma:increasing} also holds for non-Fibonacci Markoff $m$-triples. If we define
$$\tilde{m}(a,b,c)=a^2+b^2+c^2-3abc$$
so that $m(a,b,c)=\tilde{m}(F(a),F(b),F(c))$, we can repeat the proof of the previous Lemma to show that if $a\le a'\le b$ and $b\le b'\le c$, then
$$\tilde{m}(a,b,c)\ge \tilde{m}(a',b',c)\, .$$
\end{remark}

Now, we will bound the maximal element $c'$ in a couple of triples with $m(a,b,c)=m(a',b',c')$ and $c'\le c$.

\begin{lemma}
\label{lemma:KaramataBound}
Let $m>0$. Let $A,C\ge 2$ and $t\ge 1$ be three integers. If $A\le a\le b\le c=a+b+t$ and $c\ge C,$ then
$$L_{A,t,C} \frac{1}{5} \varphi^{2c}\le m(a,b,c) \le U_{A,t,C} \frac{1}{5} \varphi^{2c}, $$
where
\begin{gather}
\label{eq:Latc}
L_{A,t,C}=1-\frac{3}{\sqrt{5}} \varphi^{-t}+\left(1-\frac{3}{\sqrt{5}}\varphi^t\right)\left(\varphi^{-2t-2A}+\varphi^{2A-2C}\right)-\left(6+\frac{3}{\sqrt{5}}\varphi^t+\frac{9}{\sqrt{5}}\right)\varphi^{-2C},\\
\label{eq:Uatc}
U_{A,t,c}=1-\frac{3}{\sqrt{5}} \varphi^{-t}+\left(1+\frac{3}{\sqrt{5}}\varphi^t\right)\left(\varphi^{-2t-2A}+\varphi^{2A-2C}\right)+9\varphi^{-2C}.
\end{gather}
\end{lemma}

\begin{proof}
Using Binet's formula (\ref{eq:Binet}) and taking into account that $\varphi\bar{\varphi}=-1$, we deduce that
$$F(n)^2=\frac{1}{5}\varphi^{2n}+\frac{1}{5}\varphi^{-2n}-\frac{2}{5}(-1)^n<\frac{1}{5}\varphi^{2n}+\frac{3}{5}.$$
Thus
\begin{multline*}
m(a,b,c)=F(c)^2+F(c-t-a)^2+F(a)^2-3F(c)F(c-t-a)F(a)\\
<\frac{1}{5}\varphi^{2c}+\frac{1}{5}\varphi^{2c-2t-2a}+\frac{1}{5}\varphi^{2a}+\frac{9}{5}-\frac{3}{5\sqrt{5}} (\varphi^{c}-\bar{\varphi}^{c})(\varphi^{c-t-a}-\bar{\varphi}^{c-t-a})(\varphi^{a}-\bar{\varphi}^{a})\, .
\end{multline*}
As $c>t$ and $\varphi\bar{\varphi}=-1$, it follows that
\begin{multline*}
    (\varphi^{c}-\bar{\varphi}^{c})(\varphi^{c-t-a}-\bar{\varphi}^{c-t-a})(\varphi^{a}-\bar{\varphi}^{a})\ge
    (\varphi^{c}-\varphi^{-c})(\varphi^{c-t-a}-\varphi^{a-c+t})(\varphi^{a}-\varphi^{-a})\\
    =\varphi^{2c-t}-\varphi^{2c-t-2a}-\varphi^{2a+t}+\varphi^t-\varphi^{-t}+\varphi^{-2a-t}+\varphi^{2a-2c+t}-\varphi^{-2c+t}\\
    > \varphi^{2c-t}-\varphi^{2c-t-2a}-\varphi^{2a+t} +1-\varphi^{-2c+t}> \varphi^{2c-t}-\varphi^{2c-t-2a}-\varphi^{2a+t}\,.
\end{multline*}
Therefore,
\begin{multline*}
    m(a,b,c)< \frac{1}{5}\left(1-\frac{3}{\sqrt{5}}\varphi^{-t}\right)\varphi^{2c}+\frac{1}{5}\varphi^{2c-2t-2a}+\frac{1}{5}\varphi^{2a'}+\frac{9}{5}-\frac{3}{5\sqrt{5}}(\varphi^{2c-t}-\varphi^{2c-t-2a}-\varphi^{2a+t})
    \\=\frac{1}{5}\left(1-\frac{3}{\sqrt{5}}\varphi^{-t}\right)\varphi^{2c}+\frac{1}{5}\left(1+\frac{3}{\sqrt{5}}\varphi^t\right) ( \varphi^{2c-2t-2a}+\varphi^{2a})+\frac{9}{5}\, .
\end{multline*}
As $\varphi^x$ is a convex function and $a\ge A$, by applying Karamata's inequality \cite{K}, we get
\begin{equation}
\label{eq:karamata}
\varphi^{2c-2a}+\varphi^{2a}\le \varphi^{2c-2A}+\varphi^{2A}\, .
\end{equation}
Since the factor $1+\frac{3}{\sqrt{5}}\varphi$ is positive and $c\ge C$, then
\begin{multline*}
    m(a,b,c)< \frac{1}{5}\left(1-\frac{3}{\sqrt{5}}\varphi^{-t}\right)\varphi^{2c}+\frac{1}{5}\left(1+\frac{3}{\sqrt{5}}\varphi^t\right) ( \varphi^{2c-2t-2A}+\varphi^{2A})+\frac{9}{5}\\
    = \left(1-\frac{3}{\sqrt{5}}\varphi^{-t}+\left(1+\frac{3}{\sqrt{5}}\varphi^t\right) ( \varphi^{-2t-2A}+\varphi^{2A-2c})+9\varphi^{-2c}\right)\frac{1}{5}\varphi^{2c}\le U_{A,t,C} \frac{1}{5} \varphi^{2c}\, .
\end{multline*}
Analogously,
$$F(n)^2=\frac{1}{5}\varphi^{2n}+\frac{1}{5}\varphi^{-2n}-\frac{2}{5}(-1)^n>\frac{1}{5}\varphi^{2n}-\frac{2}{5}$$
and, since $c>t$
\begin{multline*}
    (\varphi^{c}-\bar{\varphi}^{c})(\varphi^{c-t-a}-\bar{\varphi}^{c-t-a})(\varphi^{a}-\bar{\varphi}^{a})\le
    (\varphi^{c}+\varphi^{-c})(\varphi^{c-t-a}+\varphi^{a-c+t})(\varphi^{a}+\varphi^{-a})=\\
    \varphi^{2c-t}+\varphi^{2c-t-2a}+\varphi^{2a+t}+\varphi^t+\varphi^{-t}+\varphi^{-2a-t}+\varphi^{2a-2c+t}+\varphi^{-2c+t}
    <\varphi^{2c-t}+\varphi^{2c-t-2a}+\varphi^{2a+t} +\varphi^t+3
\end{multline*}
so
\begin{multline*}
m(a,b,c)> \frac{1}{5}\left(1-\frac{3}{\sqrt{5}}\varphi^{-t}\right)\varphi^{2c}+\frac{1}{5}\varphi^{2c-2t-2a}+\frac{1}{5}\varphi^{2a'}-\frac{6}{5}-\frac{3}{5\sqrt{5}}(\varphi^{2c-t}+\varphi^{2c-t-2a}+\varphi^{2a+t}+\varphi^t+3)
    \\=\frac{1}{5}\left(1-\frac{3}{\sqrt{5}}\varphi^{-t}\right)\varphi^{2c}+\frac{1}{5}\left(1-\frac{3}{\sqrt{5}}\varphi^t\right) ( \varphi^{2c-2t-2a}+\varphi^{2a})-\frac{6}{5}-\frac{3}{5\sqrt{5}} \varphi^t - \frac{9}{5\sqrt{5}}\, .
\end{multline*}
Using again Karamata's inequality \eqref{eq:karamata}, and taking into account that the factor $1-\frac{3}{\sqrt{5}}\varphi$ is negative and that $c\ge C$, we obtain
\begin{multline*}
    m(a,b,c)>\frac{1}{5}\left(1-\frac{3}{\sqrt{5}}\varphi^{-t}\right)\varphi^{2c}+\frac{1}{5}\left(1-\frac{3}{\sqrt{5}}\varphi\right)( \varphi^{2c-2t-2A}+\varphi^{2A})-\frac{6}{5}-\frac{3}{5\sqrt{5}}\varphi^t-\frac{9}{5\sqrt{5}}\ge L_{A,t,C}\frac{1}{5}\varphi^{2c}\, .
\end{multline*}
\end{proof}

\begin{lemma}
\label{lemma:c'c-1}
Let $m>0$. Let $(F(a),F(b),F(c))$ and $(F(a'),F(b'), F(c'))$ be two ordered minimal Markoff-Fibonacci $m$-triples with $a \le b \le c$, $a'\le b'\le c'$ and $c\ge c'$. If $a\ge 4$ and $c\ge 9$ then either $c'=c$ or $c'=c-1$.
\end{lemma}

\begin{proof}
Assume that $m(a,b,c)=m=m(a',b',c')$. From Lemma \eqref{lemma:increasing} and Lemma \ref{lemma:KaramataBound} we know that, if $L_{A,t,C}$ is the lower bound given by \eqref{eq:Latc}, then
$$m=m(a,b,c)\ge m(a,c-1-a,c)\ge L_{4,1,9} \frac{1}{5}\varphi^{2c}$$
A direct computation shows that
$$L_{4,1,9}=1-\frac{3}{\sqrt{5}}\varphi^{-1}+\left(1-\frac{3}{\sqrt{5}}\varphi\right)2\varphi^{-10}-6-\frac{3}{\sqrt{5}}\varphi^{-8}-9\varphi^{-9}>1.025 \varphi^{-4}>\varphi^{-4}$$
so
\begin{equation}
\label{eq:minimal1}
m(a,b,c)>L_{4,1,9}\frac{1}{5}\varphi^{2c}>\frac{1}{5}\varphi^{2c-4}=\frac{1}{5}\varphi^{2(c-2)}\, .
\end{equation}
On the other hand, by Lemma \eqref{lemma:increasing}, it follows
\begin{equation}
\label{eq:minimal2}
m=m(a',b',c')\le m(2,2,c')=F(c')^2-3F(c')+2 < \frac{1}{5}\varphi^{2c'}+\frac{1}{5}\bar{\varphi}^{2c'}+\frac{2}{5}(-1)^{c'}-1< \frac{1}{5} \varphi^{2c'}\, .
\end{equation}
Using equations \eqref{eq:minimal1} and \eqref{eq:minimal2} together yields $\varphi^{2(c-2)} < 5m < \varphi^{2c'}$. Thus, $c'>c-2$. As we assumed $c'\le c$, then either $c'=c$ or $c'=c-1$.
\end{proof}

\begin{lemma}\label{lemma3.5} Let $(F(a), F(b), F(c))$ and $(F(a'), F(b'), F(c'))$ be two minimal  Markoff-Fibonacci $m$-triples   with: $a'\leq b'\leq c'\leq c$, $a\leq b\leq c$, $a=2$ or $3$ and $c\ge 10$. Then,  $c'\geq c-1$ except for the case $a=3$ and $c=b+4$, where $c'\geq c-2.$
 \end{lemma}
 \begin{proof} The Markoff-Fibonacci $m$-triples satisfy 
 \begin{equation}\label{eq:(2.1)}
 m(a,b,c)=m(a',b',c').
 \end{equation} 

\noindent Dividing \eqref{eq:(2.1)} by $F(c)^2$, and arranging we obtain
\begin{multline}\label{eq:(2.2)} \frac{m(a,b,c)}{F(c)^2}-
\frac{m(a',b',c')}{F(c)^2}= 1+\frac{F(a)^2}{F(c)^2}+\frac{F(b)^2}{F(c)^2}-3\frac{F(a) F(b)} {F(c)}\\
-\left(\frac{F(c')^2}{F(c)^2}+\frac{F(a')^2}{F(c)^2}+\frac{F(b')^2}{F(c)^2}-3\frac{F(a') F(b') F(c')}{F(c)^2}\right)=0.
\end{multline} 
We also have the inequality $F(n+1)>  \sqrt{2}F(n)$, for $n\ge 2$ (see, for instance,  Remark \ref{rmk:FibonacciQuotientBounds}). First, we assume that $a=2$ and $c\geq c'\geq 5$. Then, we obtain the upper bounds $a'\leq b'\leq c'-3$ and $b\leq c-3$  by the minimality of the triples (Proposition \ref{prop:minimalBound}).

If  $c'\leq c-2$, then $a'\leq b'\leq c-5,$ which implies that $F(c)> 2 F(c')$, $F(c)> 4 \sqrt{2} F(b')\geq 4 \sqrt{2} F(a')$. As a result, $F(c)> 4 \sqrt{2} F(b)$ if we first assume $c\geq b+5$, obtaining the following lower bound for the left side of \eqref{eq:(2.2)}:
 \begin{equation}\label{eq:(2.3)}\frac{m(a,b,c)}{F(c)^2}-\frac{m(a',b',c')}{F(c)^2}>1-\frac{3}{4\sqrt{2}}-\left(\frac{1}{4}+\frac{1}{32}+\frac{1}{32}\right)>0,
 \end{equation}
 contradicting equation \eqref{eq:(2.2)}.
 Therefore, $c=b+3$, with $b\geq 2$ or $c=b+4$.  If $c=b+3$, we obtain the following lower bound for the left side of \eqref{eq:(2.2)}: 
\begin{multline*}\frac{m(a,b,c)}{F(c)^2}-\frac{m(a',b',c')}{F(c)^2} > 1+\frac{1}{F(c)^2}+\frac{F(c-3)^2}{F(c)^2}-3\frac{F(c-3)}{F(c)}-\left(\frac{1}{4}+\frac{1}{32}+\frac{1}{32}\right)\\
>1+0.2360^2-3(0.2361)-1/4-1/16>0,
\end{multline*}
for $c\geq 10$ (using Tables \ref{table:1} and \ref{table:2}), which again contradicts \eqref{eq:(2.2)}, so the case $c=b+3$ is discarded.

We will apply this technique in the following cases omitting the details for the sake of clarity in the exposition.

If $c=b+4$, then the lower bound of the left side of \eqref{eq:(2.2)} is
$$\frac{m(a,b,c)}{F(c)^2}-\frac{m(a',b',c')}{F(c)^2}>1+\frac{1}{F(c)^2}+\frac{F(c-4)^2}{F(c)^2}-3\frac{F(c-4)}{F(c)}-\frac{5}{16}\geq 0.16>0$$ for $c\geq 5$, contradicting \eqref{eq:(2.2)}, so the case $c=b+4$ is not possible.

The case $c'\leq 4$ is discarded in the following way: if $c'\leq 4$, then $F(c')\leq F(4)=3$. Since $(F(a'), F(b'), F(c'))$ is minimal, we must have $F(a')=F(b')=1$  and $F(c')=3$. But, in this case, $m=2$ and, for this value of $m$, there are not two minimal Markoff-Fibonacci $m$-triples, a contradiction. As a result, $c'\geq c-1$ in the case $a=2$, as desired.

Now, consider the case $c'\leq c-2$ with $a=3$. Without loss of generality, we can assume that $c\geq 7$, $c'\geq 5,$ as in the previous case. On the other hand, since $(2, F(b), F(c))$ is minimal, then $b\leq c-4$ by Proposition \ref{prop:minimalBound}. 

If $c\geq b+5$, we obtain the following lower bound for the left side of \eqref{eq:(2.2)} 

$\frac{m(a,b,c)}{F(c)^2}-\frac{m(a',b',c')}{F(c)^2}>1-6\frac{F(c-5)}{F(c)}-\frac{5}{16}>0.1>0$ for $c\geq 8$, which again contradicts \eqref{eq:(2.2)}.
This implies that $c'\geq c-1$ for $c\geq b+5$. 
But if $c=b+4\geq 7$ and $c'\leq c-3$, then $a'\leq b' \leq c-6$. As a result, we obtain a lower bound for the left side of \eqref{eq:(2.2)}
$$\frac{m(a,b,c)}{F(c)^2}-\frac{m(a',b',c')}{F(c)^2}\geq 1+\frac{4}{F(c)^2}+\frac{F(c-4)^2}{F(c)^2}-6\frac{F(c-4)}{F(c)}-\left(\frac{F(c-3)^2}{F(c)^2}+2\frac{F(c-6)^2}{F(c)^2}\right)\geq 0.05>0$$ for $c\geq 7$, obtaining a contradiction. Thus $c'\geq c-2$ in this case, concluding the lemma. 

\end{proof}

\begin{lemma}
\label{lemma:c'c-1part2}
Let $(F(a), F(b), F(c))$ and $(F(a'),F(b'),F(c'))$ be two ordered minimal Markoff $m$-triples with $m>0$ such that $c\geq c'$. If $a=2$ or $a=3$ and $c\ge 10$ then $c'=c$ or $c'=c-1$.
\end{lemma}

\begin{proof}
According to Lemma \ref{lemma3.5}, we can assume without loss of generality that $c=b+4$ and $c'\geq c-2$. 
Let us consider the case $c'=c-2$. The two $m$-triples are $(F(a), F(b), F(b+4))$ and $(F(a'), F(b'), F(b+2))$. The equation relating both is 
\begin{equation}
\label{eq:0.1}
F(a)^2+F(b)^2+F(b+4)^2+3F(a')F(b')F(b+2)=F(a')^2+F(b')^2+F(b+2)^2+3F(a)F(b)F(b+4).
\end{equation}
In the case $F(a)=1,$ the above equation reduces to
$$1+F(b)^2+F(b+4)^2-F(b+2)^2+3F(a')F(b')F(b+2)=F(a')^2+F(b')^2+3F(b)F(b+4)$$
or
\begin{equation}
\label{eq:0.2}
1+F(b)^2+F(b+3)\left(F(b+4)+F(b+2)\right)+3F(a')F(b')F(b+2)=F(a')^2+F(b')^2+3F(b)F(b+4).
\end{equation}
Note that for any ordered minimal $m$-triple  $(x,y,z)$, it follows from \cite[Lema 2.2]{SC} that  $x^2+y^2\le m<z^2$ and hence
\begin{equation}
\label{eq:0.2b}
F(a')^2+F(b')^2-3F(a')F(b')F(b+2)=m-F(b+2)^2<0.   
\end{equation}
Moreover, $3F(b)F(b+4)<F(b+3)F(b+4)$ (see Remark \ref{rmk:FibonacciQuotientBounds}) and thus
equation \eqref{eq:0.2} does not hold and hence neither does \eqref{eq:0.1}.

Next, using $F(b+4)=5F(b)+3F(b-1)$ and $F(b+2)=2F(b)+F(b-1)$ in \eqref{eq:0.1}, we get 
$$F(a)^2+F(b)^2+25F(b)^2+9F(b-1)^2+30F(b)F(b-1)+3F(a')F(b')F(b+2)=$$
$$F(a')^2+F(b')^2+
4F(b)^2+F(b-1)^2+4F(b)F(b-1)+15F(a)F(b)^2+9F(a)F(b)F(b-1)$$
which in simplification gives
\begin{multline}
\label{eq:0.3}
  F(a)^2+8F(b-1)^2+3F(a')F(b')F(b+2)\\=F(a')^2+F(b')^2+F(b)^2(15F(a)-22)+F(b)F(b-1)(9F(a)-26).  
\end{multline}
In the case $F(a)=2,$ equation \eqref{eq:0.3} reduces to 
$$4+8F(b-1)^2+3F(a')F(b')F(b+2)=F(a')^2+F(b')^2+8F(b)^2-8F(b)F(b-1)$$
giving
\begin{multline*}4+8F(b)F(b-1)+3F(a')F(b')F(b+2)=F(a')^2+F(b')^2+8(F(b)^2-F(b-1)^2)\\
=F(a')^2+F(b')^2+8F(b+1)F(b-2),
\end{multline*}
consequently,
$$4+3F(a')F(b')F(b+2)=F(a')^2+F(b')^2+8(F(b+1)F(b-2)-F(b)F(b-1))=F(a')^2+F(b')^2\pm 8, $$
where we applied Vajda's identity with $n=b-2$, $i=2$ and $j=1$ (see Remark \ref{rmk:Vadja}).
Considering that, by equation \eqref{eq:0.2b}, $F(a')^2+F(b')^2<3F(a')F(b')F(b+2)$, the above reduces to 
$$3F(a')F(b')F(b+2)=F(a')^2+F(b')^2+4,$$
which is not possible as
$$3F(b')F(b+2)\le 3F(a')F(b')F(b+2)=F(a')^2+F(b')^2+4\le 2F(b')^2+4.$$
Therefore
$$3F(b+2)\le 2F(b')+\frac{4}{F(b')}\le 2F(b')+4< 2F(c')+4=2F(b+2)+4$$
giving $F(b+2)< 4,$ a contradiction as $b=c-4\ge 3$.

\end{proof}

\begin{lemma}
\label{lemma:sumab}
Let $m>0$. Let $(F(a),F(b),F(c))$ and $(F(a'),F(b'), F(c'))$ be two ordered minimal Markoff-Fibonacci $m$-triples with $a \le b \le c$, $a'\le b'\le c'$ and $c'=c-1$. If $c\ge 7$ then $a+b=c-1$.
\end{lemma}

\begin{proof}
By Proposition \ref{prop:minimalBound}, $a+b\le c-1$. Suppose that $a+b\le c-2$. Let us prove that in this case $m(a,b,c)>m(a',b',c')$. This will lead to a contradiction, proving that $a+b=c-1$. We will work analogously to Lemma \ref{lemma:c'c-1}. The same proof of equation \eqref{eq:minimal2} from that lemma shows that
$$m(a',b',c')< \frac{1}{5} \varphi^{2c'}= \frac{1}{5}\varphi^{2c-2}.$$
On the other hand, by Lemma \ref{lemma:increasing} and Lemma \ref{lemma:KaramataBound}, since $b\le c-2-a<c$, $a\ge 2$ and $c\ge 7,$ then
$$m(a,b,c)\ge m(a,c-2-a,c)>L_{2,2,7}\frac{1}{5}\varphi^{2c},$$
where $L_{A,t,C}$ is the constant defined at equation \eqref{eq:Latc}, which admits the following lower bound.
$$L_{2,2,7}=1-\frac{3}{\sqrt{5}}\varphi^{-2}+\left(1-\frac{3}{\sqrt{5}}\varphi^2\right)(\varphi^{-8}+\varphi^{-10})-\left(6+\frac{3}{\sqrt{5}}\varphi^2+\frac{9}{\sqrt{5}}\right)\varphi^{-14}>1.04\varphi^{-2}>\varphi^{-2}.$$
Consequently,
$$m(a,b,c)> L_{2,2,7}\frac{1}{5}\varphi^{2c}> \frac{1}{5}\varphi^{2c-2}>m(a',b',c')\, .$$
\end{proof}

\begin{lemma}
\label{lemma:suma'b'}
Let $m>0$. Let $(F(a), F(b), F(c))$ and $(F(a'),F(b'),F(c'))$ be two ordered minimal Markoff $m$-triples such that $c'=c-1$ and $(a,a')\ne (2,2)$. Then, if $c\ge 19$ the following hold.
\begin{itemize}
\item If $a\ne 2, 4$, then $a'+b'=c'-1$.
\item If $a=4$, then either $a'+b'=c'-1$ or $a'+b'=c'-2$.
\item If $a=2$, then $a'+b'+1\leq c'\leq a'+b'+5$.
\end{itemize}
\end{lemma}

\begin{proof}
First, assume that $a'\geq 3$. For two such  triples, we have that $c'=c-1$ (Lemmas \ref{lemma:c'c-1}, \ref{lemma3.5} and \ref{lemma:c'c-1part2}), $c=a+b+1$ (Lemma \ref{lemma:sumab}) and $c'\geq a'+b'+1$ (Proposition \ref{prop:minimalBound}).
Since $b \geq a$, if we first assume that $a \geq 6$, then  $c \geq b + 7$ and $c\geq 2 a+1$, so $a\leq \lfloor{\frac{c-1}{2}}\rfloor$ and we obtain the following upper bound for the left
side of \eqref{eq:(2.2)}:

\begin{multline}\label{eq:(2.9)}\frac{m(a,b,c)}{F(c)^2}-\frac{m(a',b',c')}{F(c)^2} \leq 1+\frac{F(c-7)^2}{F(c)^2}+\frac{F(\lfloor\frac{c-1}{2}\rfloor)^2}{F(c)^2}-3\frac{F(a) F(b)}{F(c)}\\-\left(\frac{F(c-1)^2}{F(c)^2}-\frac{3F(a')F(b')}{F(c)}\frac{F(c-1)}{F(c)}\right)\,.\end{multline}
We derive the following bounds for $F(n)$ based on Binet's formula (\ref{eq:Binet}) 
$$F(n)\leq \frac{1}{\sqrt{5}} \left(\varphi^n+\frac{1}{\varphi^n}\right), \quad F(n)\leq \frac{1}{\sqrt{5}}\left(\varphi^n+1\right), \quad
F(n)\geq \frac{1}{\sqrt{5}} \left(\varphi^n-\frac{1}{\varphi^n}\right),\quad F(n)\geq \frac{1}{\sqrt{5}}\left(\varphi^n-1\right).$$
These bounds allow us to establish constraints for the crossed terms of \eqref{eq:(2.9)} assuming that $c'\geq a'+b'+2$ (and then $c'\geq b'+5$):
\begin{equation*}
3\frac{F(a) F(b)}{F(c)} \geq \frac{3}{\sqrt{5}}\frac{(\varphi^a-\frac{1}{\varphi^a})(\varphi^b-\frac{1}{\varphi^b})}{\varphi^c+1}=\frac{3}{\sqrt{5}}\frac{\varphi^{a+b}-\frac{\varphi^b}{\varphi^a}-\frac{\varphi^a}{\varphi^b}+\frac{1}{\varphi^{a+b}}}{\varphi^c+1}\geq \frac{3}{\sqrt{5}}\frac{\varphi^{c-1}-\frac{\varphi^{c-7}}{\varphi^6}-1+\frac{1}{\varphi^{c-1}}}{\varphi^c+1}\,,
\end{equation*}
\begin{multline*}\frac{3F(a') F(b')}{F(c)}\leq \frac{3}{\sqrt{5}}\frac{\left(\varphi^{a'}+\frac{1}{\varphi^{a'}}\right)\left(\varphi^{b'}+\frac{1}{\varphi^{b'}}\right)}{\varphi^c-1}=\frac{3}{\sqrt{5}}\frac{\varphi^{a'+b'}+\frac{\varphi^{b'}}{\varphi^{a'}}+\frac{\varphi^{a'}}{\varphi^{b'}}+\frac{1}{\varphi^{a'+b'}}}{\varphi^{c}-1} \\\leq \frac{3}{\sqrt{5}}\frac{\varphi^{c'-2}+\frac{\varphi^{c'-5}}{\varphi^3}+2}{\varphi^{c}-1}=\frac{3}{\sqrt{5}}\frac{\varphi^{c-3}+\frac{\varphi^{c-6}}{\varphi^3}+2}{\varphi^{c}-1}\,.
\end{multline*}
\noindent Consequently, we obtain the following upper bound for the right side of \eqref{eq:(2.9)}, where we also have proceeded as in the proof of Lemma \ref{lemma3.5} to find the value of $c$.
\begin{multline}\label{eq:(2.10)}
\frac{m(a,b,c)}{F(c)^2}-\frac{m(a',b',c')}{F(c)^2} \leq1+\frac{F(c-7)^2}{F(c)^2}+\frac{F(\lfloor\frac{c-1}{2}\rfloor)^2}{F(c)^2}-\frac{3}{\sqrt{5}}\frac{\varphi^{c-1}-\varphi^{c-13}-1+\frac{1}{\varphi^{c-1}}}{\varphi^c+1}\\-\left(\frac{F(c-1)^2}{F(c)^2}-\frac{3}{\sqrt{5}}\frac{\varphi^{c-3}+\varphi^{c-9}+2}{\varphi^c-1}\,\frac{F(c-1)}{F(c)}\right)\ \leq -0.0002.
\end{multline} for $c\geq 19$.
This contradicts \eqref{eq:(2.2)} for these values of $c$.

For $a=5$, it follows that $b=c-6$ and we obtain the upper bound
\begin{multline}\frac{m(a,b,c)}{F(c)^2}-\frac{m(a',b',c')}{F(c)^2} \leq 1+\frac{F(c-6)^2}{F(c)^2}+\frac{25}{F(c)^2}-\frac{15F(c-6)}{F(c)}\\-\left(\frac{F(c-1)^2}{F(c)^2}-\frac{3}{\sqrt{5}}\frac{\varphi^{c-3}+\varphi^{c-9}+2}{\varphi^c-1}\,\frac{F(c-1)}{F(c)}\right)\leq -0.004\end{multline} for $c\geq 12,$  which contradicts \eqref{eq:(2.2)}, for $c\geq 19.$

For $a=3$, it follows that $b=c-4$ and then
\begin{multline}\label{eq:a=3}\frac{m(a,b,c)}{F(c)^2}-\frac{m(a',b',c')}{F(c)^2}\leq 1+\frac{F(c-4)^2}{F(c)^2}+\frac{4}{F(c)^2}-\frac{6F(c-4)}{F(c)}\\-\left(\frac{F(c-1)^2}{F(c)^2}-\frac{3}{\sqrt{5}}\frac{\varphi^{c-3}+\varphi^{c-9}+2}{\varphi^c-1}\,\frac{F(c-1)}{F(c)}\right) \leq -0.007\end{multline} for $c\geq 9,$ consequently, we find again a contradiction with \eqref{eq:(2.2)} for $c\geq 19.$
This implies that $c'=a'+b'+1$ for $a\geq 5$ or $a=3$ if $a' \geq 3$ and the other conditions are fulfilled, as desired.

For $a=4,$ then $b=c-5,$ obtaining 
$$\frac{m(a,b,c)}{F(c)^2}=1+\frac{F(c-5)^2}{F(c)^2}+\frac{9}{F(c)^2}-\frac{9F(c-5)}{F(c)}.$$
On the other hand, assuming that $c'\geq a'+b'+3$ we obtain that $c=c'+1\geq a'+b'+4\geq b'+7$ , $c=c'+1\geq a'+b'+4\geq 2a'+4$, therefore
$$\frac{3F(a') F(b')}{F(c)}\leq\frac{3}{\sqrt{5}}\frac{\varphi^{a'+b'}+\frac{\varphi^{b'}}{\varphi^{a'}}+\frac{\varphi^{a'}}{\varphi^{b'}}+\frac{1}{\varphi^{a'+b'}}}{\varphi^{c}-1}\leq \frac{3}{\sqrt{5}}\frac{\varphi^{c-4}+\varphi^{c-10}+2}{\varphi^{c}-1}.
$$
The upper bound for \eqref{eq:(2.2)} is
\begin{multline}\label{eq:a=4}\frac{m(a,b,c)}{F(c)^2}-\frac{m(a',b',c')}{F(c)^2}\leq 1+\frac{F(c-5)^2}{F(c)^2}+\frac{9}{F(c)^2}-\frac{9F(c-5)}{F(c)}\\-\left(\frac{F(c-1)^2}{F(c)^2}-\frac{3}{\sqrt{5}}\frac{\varphi^{c-4}+\varphi^{c-10}+2}{\varphi^c-1}\,\frac{F(c-1)}{F(c)}\right)\leq-0.008\end{multline} for $c\geq9$, 
so we get a contradiction. As a result,
$a'+b'+1\leq c'\leq a'+b'+2$.  Since the parameters are integers,
either $c'=a'+b'+1$ or $c'=a'+b'+2,$ as desired in this case.

For $a=2$, then  $b=c-3$ and
$$\frac{m(a,b,c)}{F(c)^2}=1+\frac{F(c-3)^2}{F(c)^2}+\frac{1}{F(c)^2}-\frac{3F(c-3)}{F(c)}.$$

On the other hand, assuming that $c'\geq a'+b'+6$ we find that
$c=c'+1\geq a'+b'+7\geq b'+10$, in consequence
$$\frac{3F(a') F(b')}{F(c)}\leq\frac{3}{\sqrt{5}}\frac{\varphi^{a'+b'}+\frac{\varphi^{b'}}{\varphi^{a'}}+\frac{\varphi^{a'}}{\varphi^{b'}}+\frac{1}{\varphi^{a'+b'}}}{\varphi^{c}-1}\leq \frac{3}{\sqrt{5}}\frac{\varphi^{c-7}+\varphi^{c-13}+2}{\varphi^{c}-1}\,.
$$
This gives the following upper bound for \eqref{eq:(2.2)} 
\begin{multline}\frac{m(a,b,c)}{F(c)^2}-\frac{m(a',b',c')}{F(c)^2}\leq 1+\frac{F(c-3)^2}{F(c)^2}+\frac{1}{F(c)^2}-\frac{3F(c-3)}{F(c)}\\-\left(\frac{F(c-1)^2}{F(c)^2}-\frac{3}{\sqrt{5}}\frac{\varphi^{c-7}+\varphi^{c-13}+2}{\varphi^c-1}\,\frac{F(c-1)}{F(c)}\right)\leq-0.0009\end{multline} for $c\geq 13$, a contradiction with \eqref{eq:(2.2)}. This implies that
$a'+b'+1\leq c'\leq a'+b'+5$, as desired in this case.

Assume that $a'=2$, $a\geq 3$. Now, if $c'\geq a'+b'+2$, then $c'\geq b'+4$. This implies that $b'\leq c-5$, so 
\begin{multline*}
\frac{m(a,b,c)}{F(c)^2}-\frac{m(a',b',c')}{F(c)^2}\leq 1+\frac{F(c-5)^2}{F(c)^2}+\frac{9}{F(c)^2}-\frac{9F(c-5)}{F(c)}
-\left(\frac{F(c-1)^2}{F(c)^2}-\frac{3 F(c-5)}{F(c)}\frac{F(c-1)}{F(c)}\right)\\ \leq -0.015\end{multline*}
for $c\geq 11$, a contradiction with \eqref{eq:(2.2)}. Hence $c'=a'+b'+1$ also in this case and we are done.
\end{proof}

Let us show that the remaining case $(a,a')=(2,2)$ from the previous Lemma is impossible.

\begin{lemma}
\label{lemma:a2a'2}
Let $(F(a), F(b), F(c))$ and $(F(a'),F(b'),F(c'))$ be two ordered minimal Markoff-Fibonacci $m$-triples such that $c'=c-1$. If $c\ge 11$, then $(a,a')\ne (2,2)$.
\end{lemma}

\begin{proof}
Assume that $m(2,b,c)=m(2,b',c-1)$ for some $b$, $b'$ and $c,$ with $c\ge 11$. By Lemma \ref{lemma:sumab}, $b=c-3$. Thus, we obtain
\begin{multline*}1+F(c-3)^2+F(c)^2-3F(c-3)F(c)=m(2,c-3,c)=\\m(2,b',c-1)=1+F(b')^2+F(c-1)^2-3F(b')F(c-1)\, .\end{multline*}
We can simplify this equality to get
\begin{multline*}
F(b')^2-3F(b')F(c-1)=F(c-3)^2+F(c)^2-F(c-1)^2-3F(c-3)F(c)\\
=F(c-3)^2+F(c+1)F(c-2)-3F(c-3)F(c)=k\, .	
\end{multline*}
By Vajda's identity (Remark \ref{rmk:Vadja}) we have $F(c+1)F(c-2)=F(c+2)F(c-3)+3(-1)^{c-3}$, so
\begin{equation*}
k=F(c-3)\left(F(c-3)+F(c+2)-3F(c)\right)+3(-1)^{c-3}
=-F(c-3)F(c-4)+3(-1)^{c-3}\, .
\end{equation*}
The equation $F(b')^2-3F(b')F(c-1)=k$ is a quadratic equation in $F(b')$. Let $p(x)=x^2-3F(c-1)x$. Since $k<0$ for all $c\ge 7$, the parabola $p(x)=k$ has two positive roots. One of them is greater than its vertex, $\frac{3}{2} F(c-1)$, which is greater than $F(c-1)$, so it cannot be $F(b')$, since $b'\le c-1$. We will prove that the other root, belonging to the interval $(0,\frac{3}{2}F(c-1))$ is not a Fibonacci number by showing that there exists a Fibonacci number $F(b^*)<F(c-1)$ such that $p(F(b^*-1))>k>p(F(b^*))$. Since the function $p(x)$ is strictly decreasing in the interval $(0,F(c-1))$, this will prove that the root belongs to the open interval $(F(b^*-1),F(b^*))$ and, therefore, that it cannot be a Fibonacci number. Concretely, we will show that for all $c\ge 11$
$$p(F(c-9))>k>p(F(c-8))\, .$$
Dividing both sides of the equation by $F(c-1)^2$, this inequality is equivalent to 
\begin{equation}
\label{eq:a2a'2eq1}
    \frac{F(c-9)^2}{F(c-1)^2}-3\frac{F(c-9)}{F(c-1)} > -\frac{F(c-3)}{F(c-1)}\frac{F(c-4)}{F(c-1)}+\frac{3(-1)^{c-3}}{F(c-1)^2} > \frac{F(c-8)^2}{F(c-1)^2}-3\frac{F(c-8)}{F(c-1)}\, .
\end{equation}
By Remark \ref{rmk:FibonacciQuotientBounds}, we can bound above and below the left-hand side, right-hand side and middle part of the previous equation as follows. For each $c\ge 11,$ the following holds
$$\frac{F(c-9)^2}{F(c-1)^2}-3\frac{F(c-9)}{F(c-1)}\ge -0.0671, \qquad \frac{F(c-8)^2}{F(c-1)^2}-3\frac{F(c-8)}{F(c-1)}\le -0.0998.$$
\begin{multline*}
    -0.0891\ge
    -\frac{F(c-4)}{F(c-1)}\frac{F(c-3)}{F(c-1)}+\frac{3}{F(c-1)^2}\ge
    -\frac{F(c-4)}{F(c-1)}\frac{F(c-3)}{F(c-1)}+\frac{3(-1)^{c-3}}{F(c-1)^2}\\
    \ge -\frac{F(c-4)}{F(c-1)}\frac{F(c-3)}{F(c-1)}-\frac{3}{F(c-1)^2}\ge -0.0913\,.
\end{multline*}
As $-0.0671>-0.0891$ and $-0.0913>-0.0998$, the lemma follows.
\end{proof}

We can now combine all the previous results to show that no pair of minimal Markoff-Fibonacci $m$-triples can exist when $c'<c,$ if $c$ is big enough.

\begin{lemma}
\label{lemma:c=c'}
Let $(F(a), F(b), F(c))$ and $(F(a'),F(b'),F(c'))$ be two ordered minimal Markoff-Fibonacci $m$-triples. If $c\ge 20$ then $c=c'$.
\end{lemma}

\begin{proof}
As a consequence of Lemma \ref{lemma:c'c-1} and Lemma \ref{lemma:c'c-1part2}, we know that either $c'=c$ or $c'=c-1$. It is therefore enough to prove that the case $c'=c-1$ is impossible. Assume that $c'=c-1$. By Lemma \ref{lemma:sumab}, we know that $a+b=c-1$. Due to Lemma \ref{lemma:a2a'2}, we know that $(a,a')\ne (2,2)$, so we can apply Lemma \ref{lemma:suma'b'} and obtain that $a'+b'=c'-1$ if $a\ne 4$ and that $a'+b'=c'-1$ or $a'+b'=c'-2$ if $a=4$ or that $c-6=c'-5\le a'+b'\le c'-1=c-2$ if $a=2$.

Let us first analyse the case where $a\ge 3$ and $a'+b'=c'-1=c-2$. Suppose that
$$m(a,c-1-a,c)=m(a',c-2-a',c-1)\, .$$
Applying Lemma \ref{lemma:KaramataBound}, we have that, as $a\ge 3$ and $c\ge 9$, then
\begin{gather*}
m(a,c-1-a,c)>L_{3,1,9}\frac{1}{5}\varphi^{2c}\\
m(a',c-2-a',c-1)<U_{2,1,8}\frac{1}{5}\varphi^{2(c-1)}=\varphi^{-2}U_{2,1,8} \frac{1}{5}\varphi^{2c}
\end{gather*}
A direct computation shows that $L_{3,1,9}>0.14>0.139>\varphi^{-2}U_{2,1,8}$, so $m(a,c-1-a-c)>m(a',c-2-a',c-1)$ and, therefore, the case $a'+b'=c'-1$ is impossible for $a\ge 3$. By Lemma \ref{lemma:suma'b'}, the only two remaining cases to prove the result are the following: either $a=2$ and $c-6=c'-5\le a'+b'\le c'-1 = c-2$ or $a=4$ and $a'+b'=c'-2=c-3$.

Let us begin analysing the $a=2$ case. Here, we find
\begin{multline*}
    m(a,c-1-a,c)=m(2,c-3,c)=1+F(c-3)^2+F(c)^2-3F(c-3)F(c)\\>1+\frac{1}{5}\varphi^{2c}+\frac{1}{5}\varphi^{2c-6}-\frac{4}{5}
    -\frac{3}{5}(\varphi^{c-3}-1)(\varphi^{c}-1)>\frac{1}{5}\varphi^{2c}+\frac{1}{5}\varphi^{2c-6}-\frac{3}{5}\varphi^{2c-3}\, .
\end{multline*}
Thus, for $c\ge 11$
$$\frac{5m(a,c-1-a,c)}{\varphi^{2c}}>1+\varphi^{-6}-3\varphi^{-22}>0.347\,.$$
On the other hand, using Lemma \ref{lemma:KaramataBound} for each $t=1,\ldots,5$ it follows that
$$\frac{5m(a',c-1-t-a',c-1)}{\varphi^{2c}}<\varphi^{-2}-\frac{3}{\sqrt{5}}\varphi^{-2-t}+\left(1+\frac{3}{\sqrt{5}}\varphi^t\right) ( \varphi^{-2t-6}+\varphi^{6-2c})+9\varphi^{-2c}\,.$$
The maximum of the right hand side for $1\le t\le 5$ and $c\ge 11$ is attained at $t=5$ and $c=11$, and yields
$$\frac{5\,m(a',c-1-t-a',c-1)}{\varphi^{2c}}<\varphi^{-2}-\frac{3}{\sqrt{5}}\varphi^{-7}+\left(1+\frac{3}{\sqrt{5}}\varphi^7\right)2 \varphi^{-16}+9\varphi^{-22}<0.346\,.$$
Therefore,
$$\frac{5\,m(a,c-1-a,c)}{\varphi^{2c}}>0.347>0.346>\frac{5\,m(a',c-1-t-a',c-1)}{\varphi^{2c}}$$
so the case $a=2$ is impossible. 

Now, let us examine the case $a=4$ and $c'=a'+b'+2$. In this instance, the value of $\frac{m(a,b,c)}{F(c)^2}$ can be derived as a consequence of Lemma \ref{lemma:sumab}
\begin{equation}\label{cota 0.17}\frac{m(a,b,c)}{F(c)^2}=1+\frac{F(c-5)^2}{F(c)^2}+\frac{9}{F(c)^2}-\frac{9F(c-5)}{F(c)}.\end{equation} 

\noindent On the other hand, if \(a' \geq 7\), then \(c-1=c'=a'+b'+2 \geq b'+9\), implying \(b'\leq c-10\). Also, \(c-1=c'=a'+b'+2 \geq 2a'+2\), hence \(a'\leq \lfloor \frac {c-3}{2}\rfloor\). Consequently, the left-hand side of equation \eqref{eq:(2.2)} has the following lower bound: 
\begin{multline*}
    \frac{m(a,b,c)}{F(c)^2}-\frac{m(a',b',c')}{F(c)^2}
    \geq
  1+\frac{F(c-5)^2}{F(c)^2}+\frac{9}{F(c)^2}-\frac{9F(c-5)}{F(c)}\\-\left(\frac{F(c-1)^2}{F(c)^2}+\frac{F(\lfloor \frac {c-3}{2}\rfloor)^2}{F(c)^2}+\frac{F(c-10)^2}{F(c)^2}-3\frac{F(a') F(b')}{F(c)}\frac{F(c-1)}{F(c)}\right).
\end{multline*}
The last term of this expression is bounded from below by
\begin{multline*}
    3\frac{F(a') F(b')}{F(c)}\frac{F(c-1)}{F(c)}\geq \frac{3}{\sqrt{5}}\frac{(\varphi^{a'}-1) (\varphi^{b'}-1)}{\varphi^c+1}\frac{F(c-1)}{F(c)}\\
    =\frac{3}{\sqrt{5}}\frac{(\varphi^{a'+b'}-\varphi^{a'}-\varphi^{b'}+1)}{\varphi^c+1}\frac{F(c-1)}{F(c)}
    \geq \frac{3}{\sqrt{5}}\frac{(\varphi^{c-2}-\varphi^{\lfloor \frac {c-3}{2}\rfloor}-\varphi^{c-10}+1)}{\varphi^c+1}\frac{F(c-1)}{F(c)}\,.
\end{multline*}
This gives the following lower bound for $\frac{m(a,b,c)}{F(c)^2}-\frac{m(a',b',c')}{F(c)^2}$.
 \begin{multline*}\frac{m(a,b,c)}{F(c)^2}-\frac{m(a',b',c')}{F(c)^2}\geq
  1+\frac{F(c-5)^2}{F(c)^2}+\frac{9}{F(c)^2}-\frac{9F(c-5)}{F(c)}\\
  -\left(\frac{F(c-1)^2}{F(c)^2}+\frac{F(\lfloor \frac {c-3}{2}\rfloor)^2}{F(c)^2}+\frac{F(c-10)^2}{F(c)^2}-\frac{3}{\sqrt{5}}\frac{(\varphi^{c-3}-\varphi^{\lfloor \frac {c-3}{2}\rfloor}-\varphi^{c-10}+1)}{\varphi^c+1}\frac{F(c-1)}{F(c)}\right)\geq 0.01 \ \end{multline*} for $c \geq 20$. This contradicts \eqref{eq:(2.2)}.

For the case $a'=6$, then $b'=c'-8=c-9$ and it can be seen that the value of $\frac{m(a,b,c)}{F(c)^2}-\frac{m(a',b',c')}{F(c)^2}$ has the following lower bound
\begin{multline*}\frac{m(a,b,c)}{F(c)^2}-\frac{m(a',b',c')}{F(c)^2}=1+\frac{F(c-5)^2}{F(c)^2}+\frac{9}{F(c)^2}-\frac{9F(c-5)}{F(c)}\\-\left(\frac{F(c-1)^2}{F(c)^2}+\frac{64}{F(c)^2}+\frac{F(c-9)^2}{F(c)^2}-24\frac{F(c-9) F(c-1)}{F(c)^2}\right) \geq 0.004\,\end{multline*} for $c \geq 13$, contradicting again \eqref{eq:(2.2)}.

The cases $3\leq a' \leq 5$ are managed in the same way and we omit them.

Finally, if $a'=2$, then $b'=c'-4=c-5$ and the value of $\frac{m(a,b,c)}{F(c)^2}-\frac{m(a',b',c')}{F(c)^2}$ now has the following upper bound 
\begin{multline*}\frac{m(a,b,c)}{F(c)^2}-\frac{m(a',b',c')}{F(c)^2}=1+\frac{F(c-5)^2}{F(c)^2}+\frac{9}{F(c)^2}-\frac{9F(c-5)}{F(c)}\\-\left(\frac{F(c-1)^2}{F(c)^2}+\frac{1}{F(c)^2}+\frac{F(c-4)^2}{F(c)^2}-3\frac{F(c-4) F(c-1)}{F(c)^2}\right) \leq -0.005<0\, .\end{multline*} for $c \geq 9$, a contradiction. \end{proof}

To complete our result, we need to prove that there do not exist two minimal Markoff-Fibonacci triples with the same highest element and the same $m$. In other words, let us prove now that if
$$m(a,b,c)=m(a',b',c)$$
then $a=a'$ and $b=b'$.

Let us suppose that $m(a,b,c)=m(a',b',c),$ for some $(a,b)\ne (a',b')$. Assume without loss of generality that $a\le a'$. By Lemma \ref{lemma:increasing}, if $b\le b'$ and $(a,b)\ne (a',b')$ then $m(a,b,c)<m(a',b',c)$. Thus, we must have $a\le a'\le b'<b$. On the other hand, if $a=a'$ and $b'< b,$ then again by Lemma \ref{lemma:increasing} it follows that $m(a,b',c)> m(a,b,c)$. Consequently, we can infer without loss of generality that
$$a<a'\le b'<b\le c\,.$$

\begin{lemma}
\label{lemma:a+b=a'+b'} Let $(F(a),F(b),F(c))$ and $(F(a'),F(b'),F(c))$ be two ordered minimal Markoff-Fibonacci $m$-triples with $2\le a<a'\le b'<b\le c$, then $a+b=a'+b'$.
\end{lemma}

\begin{proof}
Rearranging the equation
$$F(a)^2+F(b)^2+F(c)^2-3F(a)F(b)F(c)=F(a')^2+F(b')^2+F(c)^2-3F(a')F(b')F(c)$$
yields
\begin{equation}
\label{eq:a+b=a'+b'eq1}
    F(a)^2+F(b)^2-F(a')^2-F(b')^2=3F(c)\left(F(a)F(b)-F(a')F(b')\right)\, .
\end{equation}
The left-hand side is always positive because, as $b\ge 4$ and $a'\le b'<b$, by Remark \ref{rmk:FibonacciQuotientBounds} it follows that
$$F(b)^2>2F(b-1)^2\ge F(b')^2+F(a')^2.$$
Let us see that this is impossible if $a'+b'>a+b$. Assume that $a'+b'=a+b+t$ with $t>0$. As a result
\begin{equation*}
    \frac{F(a')F(b')}{F(a)F(b)} = \frac{(\varphi^{a'}-\bar{\varphi}^{a'})(\varphi^{b'}-\bar{\varphi}^{b'})}{(\varphi^a-\bar{\varphi}^{a})(\varphi^b-\bar{\varphi}^b)}\ge
    \frac{(\varphi^{a'}-\varphi^{-a'})(\varphi^{b'}-\varphi^{-b'})}{(\varphi^a+\varphi^{-a})(\varphi^b+\varphi^{-b})} = \frac{\varphi^{a'+b'}-\varphi^{b'-a'}-\varphi^{a'-b'}+\varphi^{-a'-b'}}{\varphi^{a+b}+\varphi^{b-a}+\varphi^{a-b}+\varphi^{-a-b}}\,.
\end{equation*}
Let $s=a+b$, then $a'+b'=s+t$. Dividing the numerator and denominator by $\varphi^s$ yields
\begin{equation*}
\frac{\varphi^{a'+b'}-\varphi^{b'-a'}-\varphi^{a'-b'}+\varphi^{-a'-b'}}{\varphi^{a+b}+\varphi^{b-a}+\varphi^{a-b}+\varphi^{-a-b}} = \frac{\varphi^t-\varphi^{t-2a'}-\varphi^{t-2b'}+\varphi^{-2s-t}}{1+\varphi^{-2a}+\varphi^{-2b}+\varphi^{-2s}}=\varphi^t\frac{1-\varphi^{-2a'}-\varphi^{-2b'}+\varphi^{-2s-2t}}{1+\varphi^{-2a}+\varphi^{-2b}+\varphi^{-2s}}\, .
\end{equation*}
As $2\le a<a'\le b'<b$, then $a\ge 2$, $a'\ge 3$, $b'\ge 3$, $b\ge 4$ and $s=a+b\ge 6$. Thus
$$\varphi^t\frac{1-\varphi^{-2a'}-\varphi^{-2b'}+\varphi^{-2s-2t}}{1+\varphi^{-2a}+\varphi^{-2b}+\varphi^{-2s}}\ge \varphi \frac{1-2\varphi^{-6}}{1+\varphi^{-4}+\varphi^{-8}+\varphi^{-12}}\ge 1.22 >1\,.$$
Therefore, $F(a')F(b')>F(a)F(b)$, which contradicts the positivity of both sides of equation \eqref{eq:a+b=a'+b'eq1}.

Thus, we must have $a+b\ge a'+b'$. Assume that $a'+b'=a+b-t$ with $t>0$ and let $s=a+b$ as before. Analogously to the previous case, 
\begin{multline*}\frac{F(a')F(b')}{F(a)F(b)} \le \frac{(\varphi^{a'}+\varphi^{-a'})(\varphi^{b'}+\varphi^{-b'})}{(\varphi^a-\varphi^{-a})(\varphi^b-\varphi^{-b})}=\varphi^{-t} \frac{1+\varphi^{-2a'}+\varphi^{-2b'}+\varphi^{-2s-2t}}{1-\varphi^{-2a}-\varphi^{-2b}+\varphi^{-2s}}
    \le\\ \varphi^{-1}\frac{1+2\varphi^{-6}+\varphi^{-14}}{1-\varphi^{-4}-\varphi^{-8}}<0.83<\frac{8}{9}\,.\end{multline*}
    
\noindent As a consequence, it follows that
$$1-\frac{F(a')F(b')}{F(a)F(b)}>1-\frac{8}{9}=\frac{1}{9}\ge \frac{1}{9F(a)^2}\, .$$
Multiplying both sides by $3F(a)F(b)F(c),$ yields
$$3F(c)\left(F(a)F(b)-F(a')F(b')\right)>\frac{F(c)F(b)}{3F(a)}\, .$$
As we assumed that $(F(a),F(b),F(c))$ is minimal, then $F(c)\ge 3F(a)F(b)$ therefore
$$3F(c)\left(F(a)F(b)-F(a')F(b')\right)>\frac{F(c)F(b)}{3F(a)}\ge F(b)^2>F(b)^2-F(b')^2+F(a)^2-F(a')^2\, .$$
This contradicts equation \eqref{eq:a+b=a'+b'eq1}, so $a+b$ cannot be greater than $a'+b'$. Thus $a+b=a'+b'$.
\end{proof}

\subsection{Proof of the main theorem}

Finally, we combine all the previous results to establish the main theorem of the paper (Theorem \ref{thm:intro}), proving first an intermediary proposition.


\begin{proposition}
\label{prop:minimal}
For each $m>0,$ except $m=21$, there exists at most one minimal Markoff-Fibonacci $m$-triple. For $m=21,$ there exist exactly two minimal Fibonacci triples: $(F(3),F(3),F(7))$ and $(F(2),F(3),F(6))$.
\end{proposition}

\begin{proof}
Let $(F(a),F(b),F(c))$ and $(F(a'),F(b'),F(c'))$ be a pair of ordered minimal Markoff-Fibonacci $m$-triples with $2\le a\le b\le c$, $2\le a'\le b'\le c'$ contradicting the proposition. Assume without loss of generality that $c\ge c'$. From the computational verification stated in Lemma \ref{lemma:computationalCheck}, we know that any counterexample to this theorem must have $c\ge 20$. By Lemma \ref{lemma:c=c'} it follows that $c=c'$. Moreover, by Lemma \ref{lemma:a+b=a'+b'}, we must have $a+b=a'+b'$.  Taking $n=a$, $i=b'-a$ and $j=b-b'=a'-a$ in Vajda's identity (Remark \ref{rmk:Vadja}), we can transform equation \eqref{eq:a+b=a'+b'eq1} into
\begin{equation*}
F(a)^2+F(b)^2-F(a')^2-F(b')^2=3F(c)\left(F(a)F(b)-F(a')F(b')\right) = (-1)^{a+1} 3F(c)F(b'-a)F(b-b')\, ,
\end{equation*}

From the proof of Lemma \ref{lemma:a+b=a'+b'}, we know that the left-hand side of this equality is positive, therefore $a$ is odd, and then
\begin{equation}
    \label{eq:thmminimaleq1}
    F(a)^2+F(b)^2-F(a')^2-F(b')^2=3F(c) F(b'-a)F(b-b')\,.
\end{equation}
However, using Lemma \ref{lemma:F(n+m)<=3F(n)F(m)}, we know that
$$F(b)\le 3F(b')F(b-b') \le 9F(a)F(b'-a)F(b-b')\, .$$
Multiplying by $F(b)$ and using minimality, $3F(a)F(b)\le F(c)$, yields
$$F(b)^2\le 3F(c)F(b'-a)F(b-b')$$
and, therefore,
$$F(b)^2-F(b')^2+F(a)^2-F(a')^2<F(b)^2\le 3F(c)F(b'-a)F(b-b')$$
which contradicts equation \eqref{eq:thmminimaleq1}. Consequently, there is no possible counterexample to the theorem with $c\ge 20$ and the result follows.
\end{proof}

\textit{Proof of Theorem \ref{thm:intro}}: Suppose that $m=2$. By Proposition \ref{prop:non-minimal}, the non-minimal Markoff-Fibonacci $2$-triples are $(1,F(b),F(b+2))$ for even $b>2$. By \cite{SC}, there only exists a minimal $2$-triple, which is $(1,1,3)=(1,F(2),F(4))$. Thus all  $2$-triples are given by $(1,F(b),F(b+2)),$ for even $b$.

For $m=21$, \cite{SC} showed that there exist exactly two minimal $21$-triples, which are $(1,2,8)=(F(2),F(3),F(6))$ and $(2,2,13)=(F(3),F(3),F(7))$, so both of them are Markoff-Fibonacci triples. From Proposition \ref{prop:non-minimal} we know that all the non-minimal Markoff-Fibonacci triples have $m=2$, so there are no more Markoff-Fibonacci $m$-triples for $m=21$.

Let us assume now that $m>0,\, m\ne 2$ and $m\ne 21$. Again, from Proposition \ref{prop:non-minimal}, we know that $m$ cannot admit a non-minimal Markoff-Fibonacci triple. Thus, any Markoff-Fibonacci $m$-triple must be minimal and, by Proposition \ref{prop:minimal}, one such triple must exist at most.

Finally, by Proposition \ref{prop:infinite}, we know that there is an infinite number of values of $m$ for which the $m$-Markoff equation admits a Markoff-Fibonacci $m$-triple and only two values (2 and 21) admit more than one triple. The rest admit exactly one solution which, according to the previous argument, must be a minimal triple.


\appendix
\section{Computational verifications}

This appendix contains a brief description of the algorithms used for the computational veri-\\fications of Lemma \ref{lemma:computationalCheck}, Remark \ref{rmk:FibonacciQuotientBounds} and Lemma \ref{lemma:a2a'2}. The full code can be found at \url{https://github.com/CIAMOD/markoff_fibonacci_m_triples}.

\subsection{Computational verification of Lemma \ref{lemma:computationalCheck}}
\label{section:appendix1}

The Python script \emph{"check\_minimal\_triples.py"} was used to check Lemma \ref{lemma:computationalCheck}, in which we need to prove that $m = 21$ is the only $m$-value with more than one minimal Fibonacci $m$-triple. Given the indexes of the Fibonacci
elements $(a,b,c)$, this program checks all triples up to $c=500$. The verification runs in 41 seconds on an Intel(R) Core(TM) i7-1165G7 @ 2.8GHz, and, if run up to $c=20$ instead, it completes the required verifications for Lemma \ref{lemma:computationalCheck} in a few milliseconds.

First, we generate a dictionary with the Fibonacci sequence, where the key is the index of the sequence and the value is the corresponding number. Then, we generate all possible combinations of indexes to generate Fibonacci triples. Using Proposition \ref{prop:minimalBound}, we only need to analyse triples with the minimality condition $c\ge a + b + 1$. Finally, we store in a dictionary all minimal Fibonacci triples assigned to their respective
m-value. With this, we can check if $m = 21$ is truly the only $m>0$ that has more than one minimal Markoff-Fibonacci $m$-triple.

\subsection{Computation of minimum bounds for section \ref{section:minimal}}
\label{section:appendix2}
The Python notebook\\ \emph{``markoff\_fibonacci.jpynb''} was used to provide a table for the values of $k_{N,a}$ and $K_{N,a}$ from Lemma \ref{lemma:boundFibonacciQuotient} for small values of $N$ and $a$, in which we can find lower and upper bounds for the ratio $\frac{F(n)}{F(n+a)}$ given a certain $a$ and $N$ such as $n\ge N$. These explicit bounds are then used to bound certain expressions in the proofs of some lemmata from Section \ref{section:minimal}, especially in Lemma \ref{lemma:a2a'2}.

As we dealt with exact bounds, we implemented two functions that round up and down the numbers to the $p$-th significant figure. This ensures that the bounds are still true. $p$ has been selected as small as needed by proofs of the Lemmata from Section \ref{section:minimal}.

\begin{table}[H]
\begin{tabular}{|l|c|c|c|c|c|c|c|c|c|}
\hline
\diagbox{N}{a} & 1 & 2 & 3 & 4 & 5 & 6 & 7 & 8 & 9 \\
\hline
2 & 0.5000 & 0.3333 & 0.2000 & 0.1250 & 0.07692 & 0.04761 & 0.02941 & 0.01818 & 0.01123 \\
\hline
3 & 0.6000 & 0.3750 & 0.2307 & 0.1428 & 0.08823 & 0.05454 & 0.03370 & 0.02083 & 0.01287 \\
\hline
4 & 0.6000 & 0.3750 & 0.2307 & 0.1428 & 0.08823 & 0.05454 & 0.03370 & 0.02083 & 0.01287 \\
\hline
5 & 0.6153 & 0.3809 & 0.2352 & 0.1454 & 0.08988 & 0.05555 & 0.03433 & 0.02122 & 0.01311 \\
\hline
6 & 0.6153 & 0.3809 & 0.2352 & 0.1454 & 0.08988 & 0.05555 & 0.03433 & 0.02122 & 0.01311 \\
\hline
7 & 0.6176 & 0.3818 & 0.2359 & 0.1458 & 0.09012 & 0.05570 & 0.03442 & 0.02127 & 0.01314 \\
\hline
8 & 0.6176 & 0.3818 & 0.2359 & 0.1458 & 0.09012 & 0.05570 & 0.03442 & 0.02127 & 0.01314 \\
\hline
9 & 0.6179 & 0.3819 & 0.2360 & 0.1458 & 0.09016 & 0.05572 & 0.03443 & 0.02128 & 0.01315 \\
\hline
10 & 0.6179 & 0.3819 & 0.2360 & 0.1458 & 0.09016 & 0.05572 & 0.03443 & 0.02128 & 0.01315 \\
\hline
\end{tabular}
\caption{Table of lower bounds ($k_{N,a}$)}
\label{table:1}
\end{table}

\vspace{-8pt}

\begin{table}[H]
\begin{tabular}{|l|c|c|c|c|c|c|c|c|c|}
\hline
\diagbox{N}{a} & 1 & 2 & 3 & 4 & 5 & 6 & 7 & 8 & 9 \\
\hline
2 & 0.6667 & 0.4000 & 0.2500 & 0.1539 & 0.09524 & 0.05883 & 0.03637 & 0.02248 & 0.01389 \\
\hline
3 & 0.6667 & 0.4000 & 0.2500 & 0.1539 & 0.09524 & 0.05883 & 0.03637 & 0.02248 & 0.01389 \\
\hline
4 & 0.6250 & 0.3847 & 0.2381 & 0.1471 & 0.09091 & 0.05618 & 0.03473 & 0.02146 & 0.01327 \\
\hline
5 & 0.6250 & 0.3847 & 0.2381 & 0.1471 & 0.09091 & 0.05618 & 0.03473 & 0.02146 & 0.01327 \\
\hline
6 & 0.6191 & 0.3824 & 0.2364 & 0.1461 & 0.09028 & 0.05580 & 0.03449 & 0.02132 & 0.01318 \\
\hline
7 & 0.6191 & 0.3824 & 0.2364 & 0.1461 & 0.09028 & 0.05580 & 0.03449 & 0.02132 & 0.01318 \\
\hline
8 & 0.6182 & 0.3821 & 0.2362 & 0.1460 & 0.09019 & 0.05574 & 0.03445 & 0.02129 & 0.01316 \\
\hline
9 & 0.6182 & 0.3821 & 0.2362 & 0.1460 & 0.09019 & 0.05574 & 0.03445 & 0.02129 & 0.01316 \\
\hline
10 & 0.6181 & 0.3820 & 0.2361 & 0.1460 & 0.09018 & 0.05573 & 0.03445 & 0.02129 & 0.01316 \\
\hline
\end{tabular}
\caption{Table of upper bounds ($K_{N,a}$)}
\label{table:2}
\end{table}

\vspace{-7pt}

\end{document}